\documentclass[onefignum,onetabnum]{siamart220329}

\usepackage{xcolor}

\usepackage{lipsum}
\usepackage{amsfonts}
\usepackage{graphicx}
\usepackage{epstopdf}
\usepackage{algorithmic}
\ifpdf
  \DeclareGraphicsExtensions{.eps,.pdf,.png,.jpg}
\else
  \DeclareGraphicsExtensions{.eps}
\fi

\newsiamremark{remark}{Remark}
\newsiamremark{hypothesis}{Hypothesis}
\crefname{hypothesis}{Hypothesis}{Hypotheses}
\newsiamthm{claim}{Claim}

\usepackage{bm}

\usepackage{hyperref}

\headers{Median QMC method in unanchored weighted Sobolev spaces}{Ziyang Ye, Josef Dick,  Xiaoqun Wang}
 
\title{A median QMC method for unbounded integrands over $\mathbb{R}^{s}$ in weighted unanchored Sobolev spaces\thanks{Submitted to the editors DATE.
\funding{The work of the first and third authors were supported by the National Science Foundation of China (No. 720711119). The second author was supported by the ARC grant DP220101811.}}}

\author{ Ziyang Ye\thanks{Department of Mathematical Sciences, Tsinghua University, Beijing 100084, People’s Republic
of China 
  (\email{yzy23@mails.tsinghua.edu.cn}).}
\and Josef Dick \thanks{School of Mathematics and Statistics, The University of New South Wales, Sydney, NSW 2052,
Australia  
  (\email{josef.dick@unsw.edu.au}).}
\and Xiaoqun Wang \thanks{Corresponding author. Department of Mathematical Sciences, Tsinghua University, Beijing 100084, People’s Republic
of China 
  (\email{wangxiaoqun@mail.tsinghua.edu.cn}).}}

\ifpdf
\hypersetup{
  pdftitle={Median QMC method for unbounded integrands over $\mathbb{R}^{s}$ in unanchored weighted Sobolev spaces},
  pdfauthor={Ziyang Ye\and Josef Dick\and  Xiaoqun Wang}
}
\fi

\begin{document}
\maketitle
\begin{abstract}
This paper investigates quasi-Monte Carlo (QMC) integration of Lebesgue integrable functions with respect to a density function over $\mathbb{R}^s$. We extend the construction-free median QMC rule  proposed by Goda and L'ecuyer (SIAM J. Sci. Comput., 2022) to the weighted unanchored Sobolev space of functions 
defined over $\mathbb{R}^s$
introduced by Nichols and Kuo (J. Complexity, 2014). By taking the median of $k = \mathcal{O}(\log N)$ independent randomized QMC estimators, we prove that for any $\epsilon\in (0,r-\frac{1}{2}]$, our method achieves a mean absolute error bound of $\mathcal{O}(N^{-r+\epsilon})$, where $N$ is the number of points and $r>\frac{1}{2}$ is a parameter determined by the function space. This rate matches the rate of randomly shifted lattice rules obtained via a component-by-component (CBC) construction, while our approach requires no specific CBC constructions or prior knowledge of the space’s weight structure. Numerical experiments demonstrate that our method attains an accuracy comparable to the CBC construction based method, and outperforms the Monte Carlo method.

\end{abstract}

\section{Introduction}\label{intro}

Numerous practical applications can be formulated as high-dimensional integrals. In this paper, we focus on integrals of the form 
\begin{equation}\label{integral}
    I_{\bm{\phi}} (f)=\int_{\mathbb{R}^s}f(\bm{y})\bm{\phi}(\bm{y})d\bm{y},
\end{equation}
where $f:\mathbb{R}^s\to \mathbb{R}$ is an integrable function and $\bm{\phi}$ is a density function over $\mathbb{R}^s$. We assume that $ \bm{\phi}(\bm{y})=\prod_{j=1}^{s}\phi(y_j)$ for some univariate density function $\phi$ over $\mathbb{R}$. Let $\bm{\Phi}$ be the cumulative distribution function (CDF) corresponding to $\bm{\phi}$ and let $\bm{\Phi}^{-1}$ denote the inverse of $\bm{\Phi}$. By using the inverse transformation $\bm{y} = \bm{\Phi}^{-1}(\bm{x})$, we have
\begin{equation}\label{integral2}
    I_{\bm{\phi}} (f) = \int_{[0,1]^s}f\circ\bm{\Phi}^{-1}(\bm{x})d\bm{x}.
\end{equation}
Such integrals arise from many practical problems, including option pricing in financial mathematics (see, e.g. \cite{gilbert2024theory,glassman2004Montecarlo,l2004quasi,xiao2018conditional,zhang2020quasi}) and partial differential equations (PDEs) with random coefficients in uncertainty quantification (see, e.g. \cite{graham2015quasi,herrmann2021quasi,kuo2016application,kuo2012quasi,wu2024error}). In these problems, the dimension $s$ can be quite large, making traditional numerical integration methods such as tensor-product Simpson’s rule or Gaussian quadrature infeasible due to the curse of dimensionality. 

The quasi-Monte Carlo (QMC) method is an effective quadrature technique to evaluate integrals on the unit cube $[0,1]^s$. This method approximates $I_{\bm{\phi}}(f)$ by an equal-weight quadrature rule 
\begin{equation}\label{qmc_e}
     Q_{P}(f)=\frac{1}{N}\sum_{\bm{x}\in P}f\circ\bm{\Phi}^{-1}(\bm{x}),
\end{equation}
where $P\subset[0,1]^s$ is a pre-designed set of $N$ points with low discrepancy. There are two main families of QMC point sets, namely digital nets \cite{dick2010digital,niederreiter1992random} and lattice point sets \cite{dick2022lattice,sloan1994lattice}. In this paper, we focus on a special kind of lattice rules, namely rank-1 lattice rules. 
The $N$ quadrature points in a rank-1 lattice rule are generated by a suitably chosen generating vector $\bm{z}$, where each component of $\bm{z}$ is an integer in $\{1, 2, \ldots, N-1\}$.

While such lattice rules provide efficient quadrature schemes, their point sets intrinsically include the origin $\bm{0}$. This inclusion becomes problematic when $\bm{\Phi}^{-1}(\bm{0})$ diverges to $-\infty$. Therefore, we use a randomized lattice rule to avoid the origin (with probability $1$) by adding a random shift modulo $1$ to each lattice point \cite{kuo2010randomly,nichols2014fast}. To analyze the QMC error for randomized lattice rules, the integrand is assumed to belong to a Banach space $B$. We consider the following shift-averaged worst-case error \cite{kuo2010randomly,nichols2014fast,sloan2002constructing}
\begin{equation*}
    e^{sh}(\bm{z};B)=\left(\int_{[0,1]^s} \sup_{f\in B,\Vert f\Vert_B\le 1}|Q_{P_{\bm{z},\bm{\Delta}}}(f)-I_{\bm{\phi}}(f)|^2d\bm{\Delta}\right)^{\frac{1}{2}},
\end{equation*}
where $\bm{z}$ is the generating vector and $P_{\bm{z},\bm{\Delta}}$ is the randomly shifted lattice point set with random shift $\bm{\Delta}$. Selecting a good generating vector $\bm{z}$ is crucial for obtaining a small error $ e^{sh}(\bm{z};B)$. A typical algorithm for searching for a suitable $\bm{z}$ is the component-by-component (CBC) construction algorithm \cite{korobov1959approximate,sloan2002component}.

Nichols and Kuo \cite{nichols2014fast} studied randomized lattice rules for the numerical integration problem (\ref{integral}). They introduced the weighted unanchored Sobolev space $\mathcal{F}$ with the norm 
\begin{equation*}
    \Vert f\Vert^2_{\mathcal{F}}=\sum_{u\subseteq \{1:s\}}\gamma_u^{-1}\int_{\mathbb{R}^{|u|}}\left(\int_{\mathbb{R}^{s-|u|}}\frac{\partial^{|u|}}{\partial \bm{y}_{u}}f(\bm{y})\prod_{j\in- u}\phi(y_j)d\bm{y}_{-u} \right)^{2}\prod_{j\in u}\psi_j^2(y_j)d\bm{y}_u.
\end{equation*}
Here $\{1:s\}=\{1,2,\ldots,s\}$ is a shorthand, $\frac{\partial^{|u|}}{\partial \bm{y}_{u}}f$ denotes the mixed partial derivatives of $f$ with respect to $\bm{y}_u=(y_j)_{j\in u}$, $-u$ denotes the complement of $u$ in $\{1:s\}$, and $\bm{y}_{-u}=(y_j)_{j\in -u}$. The function space is determined by the density function $\phi:\mathbb{R}\to \mathbb{R}$, the weight parameters $\gamma_u>0$ (with $\gamma_{\emptyset}=1$) and the weight functions $\psi_j:\mathbb{R}\to \mathbb{R}_+$. It is proven that by using the CBC construction, the QMC error bound is $\mathcal{O}(N^{-r+\epsilon})$, where $r$ is the decay rate of the Fourier coefficients of a function determined by $\phi$ and $\psi_j$ (see \cite[Theorem 7]{nichols2014fast} for details). However, in practice, there are some limitations in the CBC construction. Firstly, to compute the shift-averaged worst-case error within a reasonable time, it is necessary to assume that the $2^s-1$ weight parameters $\gamma_u$ adhere to a certain structure, such as product weights, order dependent weights or product and order dependent (POD) weights \cite{nichols2014fast}. Secondly, given an integrand $f$, the selection of the weight parameters and the weight functions is often quite challenging and requires a  detailed analysis of the mixed partial derivatives of $f$ (see for example \cite{gilbert2024theory,kuo2016application}). Thirdly, once the weight parameters and the weight functions are determined, one needs to run the CBC algorithm which finds a suitable generating vector. 
This involves the evaluation of (at least one) one-dimensional integrals over an unbounded domain, which contrasts with the scenario of integrands defined over $[0,1]^s$. 
We will elaborate on this distinction in detail in Section \ref{Lattice}.

More recently, a series of median-based QMC methods for the evaluation of integrals over $[0,1]^s$ have been developed in the works \cite{goda2022construction,goda2024universal,goda2024simple,pan2023super,pan2024super}. The median-based QMC methods use the median of several independent randomized QMC estimators as the final estimator of the integrand. For digital nets, the QMC points are randomized by a random linear scrambling \cite{goda2024universal,pan2023super,pan2024super}. For rank-1 lattice point sets, the QMC points are randomized by randomly choosing the generating vector \cite{goda2022construction,goda2024simple}. The advantages of the median-based QMC methods lie in the computational convenience and the ability to achieve a nearly optimal convergence rate automatically without the need of prior knowledge of the weights or the smoothness properties of the integrands. However, these works are restricted to integrals over the unit cube $[0,1]^s$. In this paper, we investigate the integrals over $\mathbb{R}^s$ within the framework of the weighted unanchored Sobolev space $\mathcal{F}$.

The method studied in this paper is an extension of the construction-free median QMC rule in \cite{goda2022construction} to the weighted unanchored Sobolev space $\mathcal{F}$. Similar to the construction-free median QMC rule, our method does not require knowledge about the weight parameters $\gamma_u$ or the weight functions $\psi_j$. For an odd integer $k>0$, we independently draw $k$ random shifts, each uniformly distributed over the unit cube $[0,1]^s$, and $k$ independent generating vectors, each uniformly distributed over the set of all admissible generating vectors. For each pair $(\bm{z}, \bm{\Delta})$ of the generating vector $\bm{z}$ and the shift $\bm{\Delta}$, we compute the corresponding QMC approximation $Q_{P_{\bm{z},\bm{\Delta}}}(f)$, and then take the median $M_k(f)$ of these $k$ approximations as our estimate for $I_{\bm{\phi}}(f)$.

Our main contribution is to prove that for the weighted unanchored Sobolev space $\mathcal{F}$ determined by $\phi,\bm{\gamma}=(\gamma_u)_{u\subseteq \{1:s\}}$ and $\bm{\psi}=(\psi_j)_{j=1}^{s}$, the error $|M_k(f)-I_{\bm{\phi}}(f)|$ obeys the following type of probabilistic bound: For a given $N$, any $\epsilon\in (0,r-\frac{1}{2}]$ and $\rho\in (0,1)$, there is a constant $c = c(r,\bm{\gamma}, \epsilon,\phi,\bm{\psi})>0$ (independent of $N$ and $k$) such that
\begin{equation*}
    \sup_{f\in\mathcal{F}, \Vert f\Vert_{\mathcal{F}}\le 1}\mathbb{P}\left[\left|M_k(f)-I_{\bm{\phi}}(f)\right|\ge c \rho^{-(\frac{1}{2}+r-\epsilon)} N^{-r+\epsilon}  \right]\le\rho^{\frac{k+1}{2}}/4,
\end{equation*}
where $r>\frac{1}{2}$ is a decay rate defined in Theorem \ref{bound} below. Building on this result, we further establish an upper bound for the mean absolute error (MAE) of $M_k(f)$. For a given $N$ and any $\epsilon\in (0,r-\frac{1}{2}]$, taking an odd integer $k \ge 4\lceil r\log_2 N \rceil-1$, there exists a constant $\widetilde{c} = \widetilde{c}(r,\bm{\gamma}, \epsilon,\phi,\bm{\psi})>0$ (independent of $N$ and $k$) such that

\begin{equation*}
    \sup_{f\in\mathcal{F}, \Vert f\Vert_{\mathcal{F}}\le 1}\mathbb{E}\left[\left|M_{k}(f)-I_{\bm{\phi}}(f)\right|\right]\le \widetilde{c} N^{-r+\epsilon}.
\end{equation*}
In words, by taking the median of $k = \mathcal{O}(\log N)$ independent randomized QMC estimators, our method achieves an MAE of $\mathcal{O}(N^{-r+\epsilon})$, which is the same as the convergence rate of the CBC construction \cite{nichols2014fast}. That is, with only $\mathcal{O}(\log N)$ independent replications, our median QMC method provides the same convergence rate as the CBC algorithm, without prior knowledge of the weight parameters and the weight functions. The proof of the probabilistic bound is similar to \cite{goda2022construction}, relying on the observation that only a small minority of generating vectors can lead to large errors. The MAE bound follows from the probabilistic bound and the fact that the weighted unanchored Sobolev space $\mathcal{F}$ can be embedded into the space of square integrable functions under the so-called \textbf{stronger condition} \cite{gilbert2022equivalence} (see also \eqref{stronger_condition} below).

The remainder of the paper is organized as follows. In Section \ref{Notations}, we recall some basic facts on weighted unanchored Sobolev spaces and randomly shifted lattice rules. In Section \ref{medianQMC}, we prove our main result for the median estimator. In Section \ref{numerical}, we conduct numerical experiments to support our theoretical findings. In Section \ref{conclusion}, we draw the conclusions of the paper.

\section{Notation and background}\label{Notations}

Let $\phi$ and $\Phi$ be the density function and the distribution function on $\mathbb{R}$, respectively. Write $\bm{\phi}(\bm{x})=\prod_{i=1}^{s}\phi(x_i)$ and $\bm{\Phi}(\bm{x})=\prod_{i=1}^{s}\Phi(x_i)$. Let
\begin{equation*}
    L_{\bm{\phi}}^2:=\left\{f:\mathbb{R}^s\to \mathbb{R}: \Vert f\Vert_{L_{\bm{\phi}}^2}^2: = \int_{\mathbb{R}^s}|f(\bm{y})|^2\bm{\phi}(\bm{y})d\bm{y}<\infty\right\}.
\end{equation*}

For any integer $N$, write
\begin{equation*}
    G_N:=\{a\in\{1:N\}:\text{gcd}(a,N)=1\}.
\end{equation*}
Then the cardinality of $G_N$ is $\varphi(N)$, where $\varphi$ is the Euler totient function. 
 
\subsection{Function space}

We consider the weighted unanchored Sobolev space $\mathcal{F}$ introduced in \cite{nichols2014fast}. For a collection of the weight parameters $\gamma_u>0$ for $u\subseteq\{1:s\}$ and the weight functions $\psi_j:\mathbb{R}\to \mathbb{R}_+$ for $j=1,\ldots,s$, the weighted unanchored Sobolev space $\mathcal{F}$ is the space of locally integrable functions on $\mathbb{R}^s$ such that the norm
\begin{equation*}
    \Vert f\Vert^2_{\mathcal{F}}=\sum_{u\subseteq \{1:s\}}\gamma_u^{-1}\int_{\mathbb{R}^{|u|}}\left(\int_{\mathbb{R}^{s-|u|}}\frac{\partial^{|u|}}{\partial \bm{y}_{u}}f(\bm{y})\prod_{j\in- u}\phi(y_j)d\bm{y}_{-u} \right)^{2}\prod_{j\in u}\psi_j^2(y_j)d\bm{y}_u
\end{equation*}
is finite. Throughout the paper we assume that for any $j = 1,2,\ldots,s$, the weight function $\psi_j$ satisfies the following \textbf{stronger condition} \cite{gilbert2022equivalence}
\begin{equation}\label{stronger_condition}
    \int_{-\infty}^{c}\frac{\Phi(y)}{\psi_j^2(y)}dy < \infty \quad\text{and}\quad \int_{c}^{\infty}\frac{1-\Phi(y)}{\psi_j^2(y)}dy< \infty\quad\text{for all finite $c$}.
\end{equation}
According to \cite[Lemma 12]{gilbert2022equivalence}, condition (\ref{stronger_condition}) ensures that $\mathcal{F}$ is a reproducing kernel Hilbert space (RKHS) embedded in $L_{\bm{\phi}}^2$, and we have 
\begin{equation}\label{L2_bound}
\Vert f\Vert_{L_{\bm{\phi}}^2}^2 \le \left(\sum_{u\subseteq \{1:s\} }\gamma_u\prod_{j\in u}C(\phi,\psi_j)\right) \Vert f\Vert^2_{\mathcal{F}},
\end{equation}
where
\begin{equation}\label{Constant_L2}
    C(\phi,\psi_j)=\int_{-\infty}^{\infty}\frac{\Phi(y)(1-\Phi(y))}{\psi^2_j(y)}dy,\quad j=1,2,\ldots,s.
\end{equation}

\subsection{Randomly shifted lattice rules and the CBC algorithm}\label{Lattice}
Randomly shifted lattice rules are a powerful class of QMC methods where the whole point set is generated by a single integer generating vector $\bm{z}\in G_N^s$ and a random shift $\bm{\Delta}\sim U[0,1]^s$.  The point set for a randomly shifted lattice rule can be expressed as
\begin{equation*}
    P_{s,\bm{z},\bm{\Delta}}:=\left\{\bm{x}_n=\left\{\frac{n\bm{z}}{N}+\bm{\Delta}\right\}\in[0,1]^s: n=0,1,\ldots,N-1 \right\},
\end{equation*}
where $\{\cdot\}$ denotes the fractional part of each component (i.e. $\{x\} = x - \lfloor x \rfloor$ for non-negative real numbers $x$).

For the weighted unanchored Sobolev space $\mathcal{F}$, one wishes to have a good generating vector $\bm{z}$ such that the shift-averaged worst-case error of the corresponding randomly shifted lattice rule, defined by 
\begin{equation}\label{d_esh}
    e^{sh}_{s,N}(\bm{z}):=\left(
    \mathbb{E}_{\bm{\Delta}}\left[
    \sup_{f\in\mathcal{F}, \Vert f\Vert_{\mathcal{F}}\le 1}
    \left|
    Q_{P_s,\bm{z},\bm{\Delta}}(f)-I_{\bm{\phi}} (f)\right|^2
    \right]
    \right)^{\frac{1}{2}},
\end{equation}
is small. Due to the lack of an effective explicit construction method for $\bm{z}$ in dimensions $s\ge 3$, we usually turn to computer-based search algorithms. The CBC construction algorithm is a typical method for searching for generating vectors, see Algorithm~\ref{alg:buildtree}.

\begin{algorithm}
\caption{CBC algorithm}
\label{alg:buildtree}
\begin{algorithmic}
\STATE Given a dimension $s\ge 1$, the number of points $N\ge 2$, the density function $\phi$ on $\mathbb{R}$, the weight parameters $\gamma_u>0$ for $u\subseteq\{1:s\}$ and the weight functions $\psi_j$ for $j=1,\ldots,s$. 

\STATE Set $z_1=1$. 
\FOR{$d=2,3,\ldots,s$} 
\STATE Choose $z_d\in G_N$ such that $e^{sh}_{d,N}(z_1,z_2,\ldots,z_{d})$ is minimized as a function of $z_d$, with $z_1,z_2,\ldots,z_{d-1}$ fixed. 
\ENDFOR

\RETURN $\bm{z}=(z_1,z_2,\ldots,z_s)$
\end{algorithmic}
\end{algorithm}
It is evident that the CBC algorithm relies on the explicit expression of $e^{sh}_{s,N}(\bm{z})$. According to \cite{nichols2014fast}, one can express $e^{sh}_{s,N}(\bm{z})$ as follows,
\begin{equation}\label{esh_com}
    [e^{sh}_{s,N}(\bm{z})]^2=\sum_{\emptyset\ne u\subseteq \{1:s\}}\frac{\gamma_u}{N}\sum_{n=0}^{N-1}\prod_{j\in u}\theta_j\left(\left\{\frac{nz_j}{N}\right\}\right),
\end{equation}
where
\begin{equation}\label{theta_j}
    \theta_j(u)=\int_{\Phi^{-1}(u)}^{\infty}\frac{\Phi(t)-u}{\psi_j^2(t)}dt+\int_{\Phi^{-1}(1-u)}^{\infty}\frac{\Phi(t)-1+u}{\psi_j^2(t)}dt-\int_{-\infty}^{\infty}\frac{\Phi^2(t)}{\psi_j^2(t)}dt.
\end{equation}

It should be noted that for general weight parameters $(\gamma_u)_{\emptyset \neq u \subseteq \{1:s\}}$ the computation of $[e^{sh}_{s,N}(\bm{z})]^2$ is intractable since one needs to iterate over all $2^s-1$ nonempty subsets $u$ of $\{1:s\}$. Consequently, the CBC method is limited to special weight structures such as product weights or POD weights. Using the fast Fourier transform, the CBC algorithm is accelerated to achieve the computational complexity of $\mathcal{O}(sN\log N)$ for product weights \cite{nuyens2006fast}, and $\mathcal{O}(sN\log N + s^2N)$ for POD weights \cite{nichols2014fast}. For the weighted unanchored Sobolev space over $\mathbb{R}^s$, we also need to estimate the values of $\theta_j(\frac{n}{N})$ for $j=1,\ldots,s$ and $n=0,\ldots,N-1$. This is an additional step in the construction, compared to the CBC construction for function spaces over $[0,1]^s$, since  the worst-case errors for function spaces over $[0,1]^s$ usually have explicit expressions. For example, for the weighted Korobov space for functions over $[0,1]^s$ with smoothness parameter $\alpha\in \mathbb{N}$ , the worst-case error has a similar form as (\ref{esh_com}), with all $\theta_j$ replaced by $\frac{(-1)^{\alpha+1}(2\pi)^{2\alpha}}{(2\alpha)!}B_{2\alpha}$ which can be explicitly computed \cite{dick2006good}. Here $B_{2\alpha}$ is the Bernoulli polynomial of degree $2\alpha$. The evaluations of $\theta_j(\frac{n}{N})$ involve numerical integration over unbounded regions. Although those values can be pre-computed, this can still be time-consuming, especially for the case where the $\theta_j$ are distinct for each $j$.

\section{Median QMC integration}\label{medianQMC}
Here we provide our median QMC rank-1 lattice rule for the weighted unanchored Sobolev space $\mathcal{F}$. For an odd integer $k$, we consider the estimator 
\begin{equation}\label{median}
    M_k(f)=\underset{1 \leq l \leq k}{\operatorname{median}} \;Q_{P_{s,\bm{Z}_l,\bm{\Delta}_l}}(f)
    =\underset{1 \leq l \leq k}{\operatorname{median}} \;\frac{1}{N}\sum_{n=1}^{N}f\circ\bm{\Phi}^{-1}\left(\left\{\frac{n\bm{Z}_{l}}{N}+\bm{\Delta}_{l}\right\}\right),
\end{equation}
where $\bm{\Delta}_1,\ldots,\bm{\Delta}_k \overset{i.i.d.}{\sim} U([0,1]^s)$,  $\bm{Z}_1,\ldots,\bm{Z}_k\overset{i.i.d.}{\sim} U(G_N^s)$, and all of the $2k$ random vectors are independent.

In order to obtain a convergence result for the median estimator $M_k(f)$, we recall a result in \cite[Lemma 5]{nichols2014fast} to express $e^{sh}_{s,N}(\bm{z})$ in terms of the Fourier coefficients of $\theta_j$.

\begin{lemma}\label{lemma_esh}
For $h\in \mathbb{Z}\setminus\{0\}$  and $j\in\{1:s\}$, let $\widehat{\theta}_j(h)$ denote the corresponding Fourier coefficient of $\theta_j$. Then we have
\begin{equation*}
    \widehat{\theta}_j(h)=\frac{1}{\pi^2h^2}\int_{\mathbb{R}}\frac{1}{\psi_j^2(t)}\sin^{2}(\pi h\Phi(t))dt.
\end{equation*}
For any $u\subseteq\{1:s\}$, $\bm{h}\in(\mathbb{Z}\setminus\{0\})^{|u|}$, let $\widehat{\theta}_u(\bm{h})=\prod_{j\in u}\widehat{\theta}_j(h_j)$. Then we have
\begin{equation}\label{esh}
    [e^{sh}_{s,N}(\bm{z})]^2=\sum_{\emptyset\ne u\subseteq \{1:s\}}\gamma_u
    \sum_{\substack{\bm{h}\in(\mathbb{Z}\setminus\{0\})^{|u|} \\ \bm{h}\cdot\bm{z}_{u}\equiv 0\ \text{(mod $N$)}}}\widehat{\theta}_u(\bm{h}).
\end{equation}
\end{lemma}

The expression (\ref{esh})  yields an upper bound for the average of the quantity $[e^{sh}_{s,N}(\bm{z})]^{2\lambda}$, where the exponent $\lambda$ varies within a specific interval.

\begin{theorem}\label{bound}
    Let $r>\frac{1}{2}$ be such that for each $j\in\{1:s\}$ we have some $C_j>0$ and $r_j\ge r$ satisfying
    \begin{equation}\label{coedecay}
        \widehat{\theta}_j(h)\le \frac{C_j}{|h|^{2r_j}} \quad\text{for all}\quad h\in \mathbb{Z}\setminus\{0\}.
    \end{equation}
    Then for any $\lambda\in (\frac{1}{2r},1]$, we have 
    \begin{equation}\label{upperbound}
        \frac{1}{\varphi(N)^s}\sum_{\bm{z}\in G_N^s}[e^{sh}_{s,N}(\bm{z})]^{2\lambda}\le \frac{1}{\varphi(N)}\sum_{\emptyset\ne u\subseteq \{1:s\} }\gamma_{u}^{\lambda}\prod_{j\in u}(2C_j^{\lambda}\zeta(2r_j\lambda))  .
    \end{equation}
\end{theorem}

\begin{proof}
    We prove (\ref{upperbound}) by induction on $s$, similarly as the proof of \cite[Theorem 7]{nichols2014fast}. When $s=1$, using (\ref{esh}), (\ref{coedecay}), and Jensen's inequality $(\sum_{k}a_k)^\lambda\le \sum_{k}a_k^\lambda$ for all nonnegative $a_k$ and $\lambda\in(\frac{1}{2r},1]$, we have 
    \begin{equation*}
    \begin{aligned}
         \frac{1}{\varphi(N)}\sum_{z\in G_N}[e^{sh}_{1,N}(z)]^{2\lambda}&= \frac{1}{\varphi(N)}\sum_{z\in G_N}\gamma_{\{1\}}^\lambda\left(\sum_{h\ne 0, N|hz}\widehat{\theta}_{1}(h)  \right)^\lambda\\
         &\le  \frac{1}{\varphi(N)}\sum_{z\in G_N}\gamma_{\{1\}}^\lambda \sum_{h\ne 0}\frac{C_1^\lambda}{|Nh|^{2\lambda r_1}}\\
         &=\gamma_{\{1\}}^\lambda\frac{2C_1^\lambda\zeta(2\lambda r_1)}{N^{2\lambda r_1}}\\
         &\le \frac{2\gamma_{\{1\}}^\lambda C_1^\lambda\zeta(2\lambda r_1)}{\varphi(N)}.
    \end{aligned}
    \end{equation*}
Suppose (\ref{upperbound}) holds for $s=d$, and we proceed to prove (\ref{upperbound}) for $s=d+1$. Again, using (\ref{esh}) and Jensen's inequality, we have 
\begin{align}
     &\frac{1}{\varphi(N)^{d+1}}\sum_{\bm{z}\in G_N^{d+1}}[e^{sh}_{d+1,N}(\bm{z})]^{2\lambda}\nonumber\\
        &=\frac{1}{\varphi(N)^{d+1}}\sum_{\bm{z}\in G_N^{d+1}}\left(
    \sum_{\emptyset\ne u\subseteq \{1:d+1\}}\gamma_u\sum_{\substack{\bm{h}\in(\mathbb{Z}\setminus\{0\})^{|u|}\nonumber \\ \bm{h}\cdot\bm{z}_{u}\equiv 0\ \text{(mod $N$)}  } }\widehat{\theta}_u(\bm{h})
    \right)^\lambda\nonumber\\
    &=\frac{1}{\varphi(N)^{d+1}}\sum_{\bm{z}\in G_N^{d+1}}\left([e^{sh}_{d,N}(\bm{z}_{\{1:d\}})]^{2 }+T_{d+1}(\bm{z}) \right)^\lambda\nonumber\\
    &\le \frac{1}{\varphi(N)^{d}}\sum_{\bm{z}\in G_N^{d}}[e^{sh}_{d,N}(\bm{z})]^{2\lambda}+\frac{1}{\varphi(N)^{d+1}}\sum_{\bm{z}\in G_N^{d+1}}T_{d+1}(\bm{z})^\lambda,\label{dp}
\end{align}
where
\begin{equation*}
    T_{d+1}(\bm{z})=\sum_{d+1\in u\subseteq\{1:d+1\}}\gamma_u\sum_{\substack{\bm{h}\in(\mathbb{Z}\setminus\{0\})^{|u|} \\ \bm{h}\cdot\bm{z}_{u}\equiv 0\ \text{(mod $N$)}  } }\widehat{\theta}_u(\bm{h}).
\end{equation*}
Using the same argument as in the proof of \cite[Theorem 7]{nichols2014fast} (where the notation $T_{d+1,s}^\lambda(z_{d+1}^{*})$ was used), we obtain that for any $\bm{w}\in G_N^d$,
\begin{equation*}
    \frac{1}{\varphi(N)}\sum_{z_{d+1}\in G_N}T_{d+1}(\bm{w},z_{d+1})^\lambda\le \frac{1}{\varphi(N)}\sum_{d+1\in u\subseteq\{1:d+1\}}\gamma_u^\lambda\prod_{j\in u}(2C_j^\lambda\zeta(2r_j\lambda)).
\end{equation*}
Thus
\begin{align}
    \frac{1}{\varphi(N)^{d+1}}\sum_{\bm{z}\in G_N^{d+1}}T_{d+1}(\bm{z})^\lambda&=\frac{1}{\varphi(N)^d}\sum_{\bm{w}\in G_N^d}\frac{1}{\varphi(N)}\sum_{z_{d+1}\in G_N}T_{d+1}(\bm{w},z_{d+1})^\lambda\nonumber\\
    &\le \frac{1}{\varphi(N)}\sum_{d+1\in u\subseteq\{1:d+1\}}\gamma_u^\lambda\prod_{j\in u}(2C_j^\lambda\zeta(2r_j\lambda)).\label{up}
\end{align}
Combining (\ref{dp}), (\ref{up}) and the induction hypothesis, we have proven (\ref{upperbound}) for $s=d+1$.
\end{proof}

\begin{corollary}\label{cor1}
    Suppose the conditions in Theorem \ref{bound} hold.
    Let $\bm{Z}$ be a random vector distributed uniformly over $G_N^s$. 
    For any $\delta\in (0,1)$ and $\lambda\in (\frac{1}{2r},1]$,  let 
    \begin{equation}\label{epsilon}
        \epsilon(\delta,\lambda)= \left[ \frac{1}{\delta\varphi(N)}\sum_{\emptyset\ne u\subseteq \{1:s\} }\gamma_{u}^{\lambda}\prod_{j\in u}(2C_j^{\lambda}\zeta(2r_j\lambda))  \right]^{\frac{1}{2\lambda}}.
    \end{equation}
    Then we have
    \begin{equation*}
        \mathbb{P}\left(e^{sh}_{s,N}(\bm{Z})\le \epsilon(\delta,\lambda)\right)\ge 1-\delta.
    \end{equation*}
\end{corollary}
\begin{proof}
    Let $g(x)=x^\lambda$. By using Markov's inequality, we have 
    \begin{equation*} 
    \mathbb{P}\left(e^{sh}_{s,N}(\bm{Z})> \epsilon(\delta,\lambda)\right)=\mathbb{P}\left(g([e^{sh}_{s,N}(\bm{Z})]^2)>g(\epsilon(\delta,\lambda)^2) \right)\le \frac{\mathbb{E}\left[g([e^{sh}_{s,N}(\bm{Z})]^2) \right]  }{g(\epsilon(\delta,\lambda)^2)}\le \delta,
    \end{equation*}
    where in the last inequality we used Theorem \ref{bound}.
\end{proof}
\begin{theorem}\label{Thm2}
    Suppose the conditions in Theorem \ref{bound} hold. Let $\bm{Z}$ and $\bm{\Delta}$ be two independent random vectors distributed uniformly over $G_N^s$ and $[0,1]^s$, respectively. Let $\epsilon(\delta,\lambda)$ be defined as in (\ref{epsilon}). Then for any $f\in \mathcal{F}, \lambda\in (\frac{1}{2r},1]$, and  $\delta_1,\delta_2\in (0,1)$, we have 
    \begin{equation}\label{prob}
        \mathbb{P}\left[\left|\frac{1}{N}\sum_{n=1}^{N}f\circ\bm{\Phi}^{-1}\left(\left\{\frac{n\bm{Z}}{N}+\bm{\Delta}\right\}\right)-I_{\bm{\phi}}(f)\right|\ge \frac{\epsilon(\delta_2,\lambda)\Vert f\Vert_{\mathcal{F}}}{\sqrt{\delta_1}}
       \right]\le \delta_1+\delta_2.
    \end{equation}
\end{theorem}

\begin{proof}

Write
\begin{align*}
    A:=\Bigg\{(\bm{z},\Delta):
    \left|\frac{1}{N}\sum_{n=1}^{N}f\circ\bm{\Phi}^{-1}\left(\left\{\frac{n\bm{z}}{N}+\Delta\right\}\right)-I_{\bm{\phi}}(f)\right|\ge \frac{\epsilon(\delta_2,\lambda)\Vert f\Vert_{\mathcal{F}}}{\sqrt{\delta_1}}  \Bigg\},
\end{align*}
and
\begin{equation*}
    X:=\left\{\bm{z}\in G_N^s :e^{sh}_{s,N}(\bm{z})> \epsilon(\delta_2,\lambda)  \right\}.
\end{equation*}
According to Corollary \ref{cor1}, $\mathbb{P}(\bm{Z}\in X)\le \delta_2$. By the law of total probability, we have
\begin{align}
    &\mathbb{P}((\bm{Z},\bm{\Delta})\in A)\nonumber\\
    &=\sum_{\bm{z}\in X}\mathbb{P}((\bm{Z},\bm{\Delta})\in A|\bm{Z}=\bm{z})\mathbb{P}(\bm{Z}=\bm{z})+\sum_{\bm{z}\in G_N^s\setminus X}\mathbb{P}((\bm{Z},\bm{\Delta})\in A|\bm{Z}=\bm{z})\mathbb{P}(\bm{Z}=\bm{z})\nonumber\\
    &\le \sum_{\bm{z}\in X}\mathbb{P}(\bm{Z}=\bm{z})+\sum_{\bm{z}\in G_N^s\setminus X}\mathbb{P}((\bm{Z},\bm{\Delta})\in A|\bm{Z}=\bm{z})\mathbb{P}(\bm{Z}=\bm{z})\nonumber\\
    &\le \delta_2+\sum_{\bm{z}\in G_N^s\setminus X}\mathbb{P}((\bm{Z},\bm{\Delta})\in A|\bm{Z}=\bm{z})\mathbb{P}(\bm{Z}=\bm{z}).\label{total}
\end{align}
Notice that for any $\bm{z}\in G_N^s\setminus X$, $e^{sh}_{s,N}(\bm{z})\le \epsilon(\delta_2,\lambda)$. Therefore, for any $\bm{z}\in G_N^s\setminus X$, we have
\begin{align}
    &\mathbb{P}((\bm{Z},\bm{\Delta})\in A|\bm{Z}=\bm{z})\nonumber\nonumber\\
    &\le \mathbb{P}\left[\left|\frac{1}{N}\sum_{n=1}^{N}f\circ\bm{\Phi}^{-1}\left(\left\{\frac{n\bm{z}}{N}+\bm{\Delta}\right\}\right)-I_{\bm{\phi}}(f)\right|\ge \frac{e^{sh}_{s,N}(\bm{z})\Vert f\Vert_{\mathcal{F}}}{\sqrt{\delta_1}}\right]\nonumber\nonumber\\
    &\le \frac{\delta_1}{[e^{sh}_{s,N}(\bm{z})]^2\Vert f\Vert_{\mathcal{F}}^2}\mathbb{E}\left[\left|\frac{1}{N}\sum_{n=1}^{N}f\circ\bm{\Phi}^{-1}\left(\left\{\frac{n\bm{z}}{N}+\bm{\Delta}\right\}\right)-I_{\bm{\phi}}(f)\right|^2 \right]\nonumber\\
    &\le \delta_1,\label{part2}
\end{align}
where in the second inequality we used Markov's inequality and in the last inequality we used (\ref{d_esh}). Combining (\ref{total}) and (\ref{part2}), we obtain (\ref{prob}).
\end{proof}

Corollary \ref{cor1} and Theorem \ref{Thm2} indicate that the vast majority of the choices of generating vectors is good for the randomized lattice rules over the function space $\mathcal{F}$. Combining Theorem \ref{Thm2} with \cite[Proposition 3.2]{goda2024universal}, we have the following bound for our median QMC estimator.

\begin{theorem}\label{Thm3}
    Suppose the conditions in Theorem \ref{bound} hold.  For an odd integer $k$, let $\bm{Z}_1,\ldots,\bm{Z}_k$ and $\bm{\Delta}_1,\ldots,\bm{\Delta}_k$ be independent random vectors such that $\bm{Z}_i \overset{i.i.d.}{\sim} U(G_N^s)$ and $\bm{\Delta}_i \overset{i.i.d.}{\sim} U([0,1]^s)$ for $i=1,\ldots,k$. Define

    \begin{equation*}
        M_k(f)=\underset{1 \leq l \leq k}{\operatorname{median}} \;\frac{1}{N}\sum_{n=1}^{N}f\circ\bm{\Phi}^{-1}\left(\left\{\frac{n\bm{Z}_{l}}{N}+\bm{\Delta}_{l}\right\}\right).
    \end{equation*}
    For any $\lambda\in (\frac{1}{2r},1]$ and any $\delta_1,\delta_2\in (0,1)$ with $\delta_1+\delta_2<\frac{1}{4}$, if $f\in \mathcal{F}$, then we have
    \begin{align}
        &\mathbb{P}\left[\left|M_k(f)-I_{\bm{\phi}}(f)\right|\ge\frac{\Vert f\Vert_{\mathcal{F}}}{\sqrt{\delta_1}}\left[ \frac{1}{\delta_2\varphi(N)}\sum_{\emptyset\ne u\subseteq \{1:s\} }\gamma_{u}^{\lambda}\prod_{j\in u}(2C_j^{\lambda}\zeta(2r_j\lambda))  \right]^{\frac{1}{2\lambda}}
        \right]\nonumber\\
        &\le 2^{k-1}(\delta_1+\delta_2)^{\frac{k+1}{2}}.\label{m_up}
    \end{align}
    
\end{theorem}

By taking $\frac{1}{2\lambda}=r-\epsilon$ and $4(\delta_1+\delta_2) = \rho$ in (\ref{m_up}), we obtain the following corollary as a simplified version of Theorem \ref{Thm3}.

\begin{corollary}\label{simple_bound}
    Let $\mathcal{F}$ be the weighted unanchored Sobolev space determined by the density function $\phi$, the weight parameters $\bm{\gamma}=(\gamma_u)_{u\subseteq\{1:s\}}$ and the weight functions $\bm{\psi}=(\psi_j)_{j=1}^{s}$. Let $r$ be defined as in Theorem \ref{bound}. Then for any odd $k\ge 3,\epsilon\in (0,r-\frac{1}{2}]$ and $\rho\in (0,1)$, there is a constant $c = c(r,\bm{\gamma}, \epsilon,\phi,\bm{\psi})>0$ (independent of $N$ and $k$) such that
    \begin{equation*}
        \sup_{f\in\mathcal{F}, \Vert f\Vert_{\mathcal{F}}\le 1}\mathbb{P}\left[\left|M_k(f)-I_{\bm{\phi}}(f)\right|\ge \frac{c(r,\bm{\gamma}, \epsilon,\phi,\bm{\psi})}{\rho^{\frac{1}{2}+r-\epsilon} }N^{-r+\epsilon}   \right]\le\rho^{\frac{k+1}{2}}/4.
    \end{equation*}
\end{corollary}
\begin{proof}
    Taking $\lambda = \frac{1}{2(r-\epsilon)}$ and $\delta_1=\delta_2 = \frac{\rho}{8}$ in (\ref{m_up}), we have $\frac{1}{2\lambda} = r-\epsilon,4(\delta_1+\delta_2) = \rho$, and 
    \begin{align}
    &\frac{\rho^{\frac{k+1}{2}}}{4} = 2^{k-1}(\delta_1+\delta_2)^{\frac{k+1}{2}}\nonumber\\
    &\ge\mathbb{P}\left[\left|M_k(f)-I_{\bm{\phi}}(f)\right|\ge\frac{\Vert f\Vert_{\mathcal{F}}}{\sqrt{\delta_1}}\left[ \frac{1}{\delta_2\varphi(N)}\sum_{\emptyset\ne u\subseteq \{1:s\} }\gamma_{u}^{\lambda}\prod_{j\in u}(2C_j^{\lambda}\zeta(2r_j\lambda))  \right]^{\frac{1}{2\lambda}}\right]\nonumber\\
    &\ge\mathbb{P}\left[\left|M_k(f)-I_{\bm{\phi}}(f)\right|\ge\frac{8^{\frac{1}{2}+r-\epsilon}\Vert f\Vert_{\mathcal{F}}}{\varphi(N)^{r-\epsilon}\rho^{\frac{1}{2}+r-\epsilon}}\left[\sum_{\emptyset\ne u\subseteq \{1:s\} }\gamma_{u}^{\lambda}\prod_{j\in u}\left(2C_j^{\lambda}\zeta\left(2r\lambda\right)\right)\right]^{r-\epsilon}\right].\label{cor_bound}
    \end{align}
    For prime $N$, we have $\varphi(N)=N-1$ and the corollary follows from (\ref{cor_bound}). For a general $N$, we know from \cite[Theorem 15]{rosser1962approximate} that
    \begin{equation*}
        \frac{1}{\varphi(N)}\le \frac{1}{N}\left[e^C\log\log N +\frac{2.50637}{\log\log N}\right]
    \end{equation*}
    for any $N\ge 3$, where $C = 0.57721\ldots$ is Euler's constant. We also note that the parameters $C_j$ are determined by $\phi,\psi_j$ and $r$. From these, the probabilistic bound follows.
\end{proof}

\begin{remark}\label{dim_indep}
    We can find that the error bound in (\ref{cor_bound}) is at the same convergence rate as that of the CBC construction in \cite{nichols2014fast}. Similarly, if the weight parameters $\{\gamma_u\}$ satisfy
    \begin{equation}\label{para_decay1}
        \sum_{|u|<\infty}\gamma_{u}^{\frac{1}{2(r-\epsilon)}}\prod_{j\in u}\left(2C_j^{\frac{1}{2(r-\epsilon)}}\zeta\left(\frac{r}{r-\epsilon}\right)\right)<\infty
    \end{equation}
    for some $\epsilon\in (0,r-\frac{1}{2}]$,
    then with high probability the median estimator can achieve an error bound of $\mathcal{O}(N^{-r+\epsilon})$ with the implied constant independent of $s$.

\end{remark}

Note that $\mathcal{F}\subset L^2_{\bm{\phi}}$ under the stronger condition (\ref{stronger_condition}). Combining Corollary \ref{simple_bound} with this embedding, we can bound the $L^1$ error for the median QMC estimator as follows.

\begin{theorem}
For any odd $k\ge 3,\epsilon\in (0,r-\frac{1}{2}]$, and $\rho\in (0,1)$, let $c(r,\bm{\gamma}, \epsilon,\phi,\bm{\psi})$ be the constant defined in Corollary \ref{simple_bound} and let $C(\phi,\psi_j)$ be the constant defined in (\ref{Constant_L2}). Then we have the MAE bound
\begin{equation}\label{L1_detail}
\sup_{ f\in\mathcal{F}, \Vert f\Vert_{\mathcal{F}\le 1}}\mathbb{E}\left[\left| 
  M_k(f)-I_{\bm{\phi}}(f)\right|\right]
\le \frac{c(r,\bm{\gamma}, \epsilon,\phi,\bm{\psi})}{\rho^{\frac{1}{2}+r-\epsilon} N^{r-\epsilon}}  + \frac{\rho^{\frac{k+1}{4}}}{2}\sqrt{k}C_1(\bm{\gamma},\phi,\bm{\psi}),
\end{equation}
where
\begin{equation*}
    C_1(\bm{\gamma},\phi,\bm{\psi}) = \sqrt{\sum_{u\subseteq \{1:s\} }\gamma_u\prod_{j\in u}C(\phi,\psi_j)}.
\end{equation*}
\end{theorem}

\begin{proof}
For any $f\in\mathcal{F}, \Vert f\Vert_{\mathcal{F}}\le 1$, let 
\begin{equation*}
    g(\bm{x}):=f\circ \bm{\Phi}^{-1}(\bm{x}),\quad\forall \bm{x}\in [0,1]^s.
\end{equation*}
According to (\ref{L2_bound}), we have $g\in L^2([0,1]^s)$ and 
\begin{equation*}
    \int_{[0,1]^s}|g(\bm{x})|^2d\bm{x} = \Vert f\Vert_{L_{\bm{\phi}}^2}^2\le \sum_{u\subseteq \{1:s\} }\gamma_u\prod_{j\in u}C(\phi,\psi_j)=C_1(\bm{\gamma},\phi,\bm{\psi})^2.
\end{equation*}
For any generating vector $\bm{z}\in G_N^s$, we have
\begin{equation*}
\begin{aligned}
&\mathbb{E}_{\bm{\Delta}}\left[\left|Q_{P_{s,\bm{z},\bm{\Delta}}}(f)-I_{\bm{\phi}}(f)\right|^2\right]\\  & = \frac{1}{N^2}\sum_{m,n=1}^{N}\mathbb{E}_{\bm{\Delta}}\left[g\left(\left\{\frac{m\bm{z}}{N}+\bm{\Delta} \right\}\right)g\left(\left\{\frac{n\bm{z}}{N}+\bm{\Delta} \right\}\right)\right]-I_{\bm{\phi}}(f)^2.
\end{aligned}
\end{equation*}
By applying the Cauchy–Schwarz inequality, we obtain that for any $1\le m,n\le N$,
\begin{equation*}
\begin{aligned}
    &\left[\int_{[0,1]^s}g\left(\left\{\frac{m\bm{z}}{N}+\bm{\Delta} \right\}\right)g\left(\left\{\frac{n\bm{z}}{N}+\bm{\Delta} \right\}\right)d\bm{\Delta}\right]^2\\
    &\le \int_{[0,1]^s}\left|g\left(\left\{\frac{m\bm{z}}{N}+\bm{\Delta} \right\}\right)\right|^2d\bm{\Delta} \int_{[0,1]^s}\left|g\left(\left\{\frac{n\bm{z}}{N}+\bm{\Delta} \right\}\right)\right|^2d\bm{\Delta}\\
    &=\left[\int_{[0,1]^s}|g(\bm{x})|^2d\bm{x}\right]^2.
\end{aligned}
\end{equation*}
Therefore, 
\begin{equation*}
\mathbb{E}_{\bm{\Delta}}\left[\left|Q_{P_{s,\bm{z},\bm{\Delta}}}(f)-I_{\bm{\phi}}(f)\right|^2\right] 
 \le \int_{[0,1]^s}|g(\bm{x})|^2d\bm{x} - I_{\bm{\phi}}(f)^2 \le C_1(\bm{\gamma},\phi,\bm{\psi})^2.
\end{equation*}
Now 
consider the event
\begin{equation*}
    A=\left\{|M_k(f)-I_{\bm{\phi}}(f)|\ge \frac{c(r,\bm{\gamma}, \epsilon,\phi,\bm{\psi})}{\rho^{\frac{1}{2}+r-\epsilon} }N^{-r+\epsilon}\right\}.
\end{equation*}
According to Corollary \ref{simple_bound},  we have
\begin{equation*}
\mathbb{P}\left[A \right]\le \rho^{\frac{k+1}{2}}/4.
\end{equation*}
Note that
\begin{equation*}
\begin{aligned}
&\mathbb{E}\left[|M_k(f)-I_{\bm{\phi}}(f)|^2\right]\\
&=\mathbb{E}\left[\left|\underset{1 \leq l \leq k}{\operatorname{median}} \;Q_{P_{s,\bm{Z}_l,\bm{\Delta}_l}}(f) -I_{\bm{\phi}}(f) \right|^2 \right]\\
&\le \mathbb{E}\left[\sum_{l=1}^{k}\left|Q_{P_{s,\bm{Z}_l,\bm{\Delta}_l}}(f)-I_{\bm{\phi}}(f) \right|^2\right]\\
&=k\mathbb{E}_{\bm{Z}_1}\left[\mathbb{E}_{\bm{\Delta}_1}\left[\left|Q_{P_{s,\bm{Z}_1,\bm{\Delta}_1}}(f)-I_{\bm{\phi}}(f) \right|^2\right]\right]\\
&\le kC_1(\bm{\gamma},\phi,\bm{\psi})^2.
\end{aligned}
\end{equation*}
Combining these inequalities, we obtain
\begin{equation*}
\begin{aligned}
&\mathbb{E}\left[\left|M_k(f)-I_{\bm{\phi}}(f)\right|\right]\\
&= \mathbb{E}\left[\left|M_k(f)-I_{\bm{\phi}}(f)\right|\bm{1}_{A^c}\right]  + \mathbb{E}\left[\left|M_k(f)-I_{\bm{\phi}}(f)\right|\bm{1}_{A}\right] \\
&\le   \frac{c(r,\bm{\gamma}, \epsilon,\phi,\bm{\psi})}{\rho^{\frac{1}{2}+r-\epsilon} N^{r-\epsilon}} + \left(\mathbb{E}\left[|M_k(f)-I_{\bm{\phi}}(f)|^2\right]\mathbb{E}\left[\bm{1}_{A}^2\right]\right)^{\frac{1}{2}}\\
&\le \frac{c(r,\bm{\gamma}, \epsilon,\phi,\bm{\psi})}{\rho^{\frac{1}{2}+r-\epsilon} N^{r-\epsilon}} + \frac{\rho^{\frac{k+1}{4}}}{2}\sqrt{k}C_1(\bm{\gamma},\phi,\bm{\psi})
.
\end{aligned}
\end{equation*}

\end{proof}

Taking $\rho^* = \frac{1}{2}$ and $k^* = 4\lceil r\log_2 N \rceil-1$, we obtain the $L^1$ convergence of the median QMC estimator.

\begin{corollary}\label{Cor_L1}
Let $\mathcal{F}$ be the weighted unanchored Sobolev space determined by the density function $\phi$, the weight parameters $\bm{\gamma}=(\gamma_u)_{u\subseteq\{1:s\}}$ and the weight functions $\bm{\psi}=(\psi_j)_{j=1}^{s}$ and let $r$ be defined as in Theorem \ref{bound}. Then for any $\epsilon\in (0,r-\frac{1}{2}]$, taking an odd integer $k \ge 4\lceil r\log_2 N \rceil-1$, there exists a constant $\widetilde{c} = \widetilde{c}(r,\bm{\gamma}, \epsilon,\phi,\bm{\psi})>0$ (independent of $N$ and $k$) such that

\begin{equation*}
    \sup_{f\in\mathcal{F}, \Vert f\Vert_{\mathcal{F}}\le 1}\mathbb{E}\left[\left|M_{k}(f)-I_{\bm{\phi}}(f)\right|\right]\le \widetilde{c}(r,\bm{\gamma}, \epsilon,\phi,\bm{\psi}) N^{-r+\epsilon}.
\end{equation*}

\end{corollary}

\begin{remark}
    If the parameters $\{\gamma_u\}$ satisfy 
    \begin{equation*}
        \sum_{|u|<\infty}\gamma_u\prod_{j\in u}C(\phi,\psi_j)<\infty
    \end{equation*}
    and the condition (\ref{para_decay1}) in Remark \ref{dim_indep} holds for some $\epsilon\in(0,r-\frac{1}{2}]$, then both $c(r,\bm{\gamma}, \epsilon,\phi,\bm{\psi})$ and $C_1(\bm{\gamma},\phi,\bm{\psi})$ in (\ref{L1_detail}) can be bounded by a number independent of the dimension $s$. In this scenario, when the odd integer $k \ge 4\lceil r\log_2 N \rceil-1$, the median estimator $M_k(f)$ can achieve an MAE bound of $\mathcal{O}(N^{-r+\epsilon})$ with the implied constant independent of $s$.
\end{remark}

\begin{remark}
While the parameter $r$ is typically unknown in practice, \cite[Table 1]{nichols2014fast} provides known ranges for $r$ corresponding to specific combinations of the density $\phi$ and the weight function $\psi_j$. For these combinations, $r$ typically satisfies $r \leq 1$. Therefore, selecting $k \ge 4\lceil \log_2 N \rceil - 1$ is sufficient in this general case to ensure that the median QMC estimator $M_k(f)$ achieves a convergence rate of $\mathcal{O}(N^{-r+\epsilon})$. Importantly, this choice holds for the typical range $r\le 1$. In particular, certain pairs of $\phi$ and $\psi_j$ (e.g., $\phi(x) = \frac{1}{\sqrt{2\pi}}e^{-\frac{x^2}{2}}$, $\psi_j(x) = e^{-\alpha_j|x|}$ with $\alpha_j>0$) yield $r$ approaching $1$, which under the choice of $k \ge 4\lceil \log_2 N \rceil - 1$, results in a convergence rate of $\mathcal{O}(N^{-1+\epsilon})$.

To handle cases where $r$ can be larger than $1$, one can employ the strategy proposed by \cite{goda2024simple}: choose $k \ge 4\lceil h(N)\log_2 N \rceil - 1$ , where $h:\mathbb{N}\to [1,\infty)$ is an arbitrary function satisfying $h(n)\to\infty$ for $n\to\infty$. The function $h$ can increase very slowly (for instance, $h(n)=\max(1,\log\log n)$).  This adjustment ensures that the median QMC estimator $M_k(f)$ achieves the convergence rate $\mathcal{O}(N^{-r+\epsilon})$, even for $r>1$.

\end{remark}

\section{Numerical experiments}\label{numerical}
In this section, we conduct numerical experiments to support our results. The first example is dedicated to investigating the empirical distribution of the shift-averaged worst-case errors. Subsequently, the second and third examples are designed to compare the performance of the median QMC method with the CBC method for option pricing and PDEs with random coefficients, respectively. It should be highlighted that although Corollary \ref{Cor_L1} suggests taking the median of $k \ge 4\lceil r \log_2 N \rceil-1$ independent QMC estimators, in Example 2 and Example 3 we have observed satisfactory results with the median of only $k = 11$ QMC estimators.

\subsection{Example 1: The distribution of the shift-averaged worst-case errors}\label{Exp1}

As our first example, we consider the weighted unanchored Sobolev space with 
\begin{equation}\label{ex1_weight}
    \phi(x)=\frac{1}{\sqrt{2\pi}}e^{-\frac{x^2}{2}}, \quad \psi_j(x) = e^{-\frac{|x|}{16}},\quad\gamma_u = \prod_{j\in u }\gamma_j, \quad \gamma_j = \frac{1}{j^2}.
\end{equation}
According to (\ref{esh_com}), the shift-averaged worst-case error of the rank-1 lattice rule with generating vector $\bm{z}$ is 
\begin{equation*}
    e_{s,N}^{sh}(\bm{z})=\left(-1+\frac{1}{N}\sum_{k=0}^{N-1}\prod_{j=1}^{s}\left[1+\gamma_j\theta_j\left(\left\{\frac{kz_j}{N}\right\}\right) \right]\right)^{\frac{1}{2}}.
\end{equation*}

We take two primes $N = 257$ and $N = 2053$, both for $s = 30$ dimensions. For those fixed $N$, we draw $20,000$ generating vectors randomly and uniformly from $\{1,2,\ldots,N-1\}^s$, and compute the corresponding $e_{s,N}^{sh}(\bm{z})$ for each generating vector. The upper panels of Figure \ref{fig1:distribution} show a histogram of the $20,000$ realizations of $\log_2 e_{s,N}^{sh}(\bm{z})$ of these cases. Interestingly, the empirical distributions exhibit a positive skewness: Most of $\log_2 e_{s,N}^{sh}(\bm{z})$ is concentrated on the left side. From these empirical distributions, we can estimate the $q$-quantiles $y_q$ of the distribution of $\log_2 e_{s,N}^{sh}(\bm{z})$. For $q = 0.75$, the corresponding empirical $q$-quantiles are $-4.8631$ for $N = 257$  and $-6.6129$ for $N = 2053$, while  for $q = 0.9$  they are $-4.4726$ for $N = 257$  and $-6.0716$ for $N = 2053$. As a benchmark we compute the shift-averaged worst-case error of the vector generated by the CBC method for both cases and mark it with a red dashed line in each histogram. We also plot the green dashed line to represent the $90$th percentile of the empirical distributions. The values of $\log_2 e_{s,N}^{sh}(\bm{z})$ for the CBC generating vectors are  $-5.7299$ for $N = 257$ and $-8.0555$ for $N = 2053$. Therefore, for both cases, more than $90\%$ of the generating vectors have the shift-averaged worst-case error less than $4$ times that of the CBC generating vector, and more than $75\%$  of the generating vectors have the shift-averaged worst-case error less than $3$ times that of the CBC generating vector. These empirical results agree with the argument that most of the generating vectors $\bm{z}$ are good for lattice rules.
\begin{remark}
    It may be of interest to compare the empirical shift-averaged worst-case errors with the theoretical bound $\epsilon(\delta,\lambda)$ established in Corollary \ref{cor1}. According to \cite[Example 5]{kuo2010randomly}, for the weighted unanchored Sobolev space determined by (\ref{ex1_weight}), the parameters $r_j$ and $C_j$ in Corollary \ref{cor1} are given by
    \begin{equation*}
        r_j=r=1-\eta,\quad C_j = C = \frac{\sqrt{2\pi}e^{\frac{1}{256\eta}}}{\pi^{2-2\eta}(1-\eta)\eta},\quad\text{for all }\eta\in(0,\frac{1}{2}).
    \end{equation*}
    Consequently, we have
    \begin{equation*}
        \inf_{\substack{\delta\in(0,1)\\r\in(\frac{1}{2},1),\lambda\in(\frac{1}{2r},1]}}\log_{2}\epsilon(\delta,\lambda)=\inf_{\substack{\eta\in(0,\frac{1}{2})\\ \lambda\in (\frac{1}{2(1-\eta)},1]}}\log_{2}\epsilon(1,\lambda)
        \approx \left\{
        \begin{aligned}
            1.2581 \quad &\text{for }N = 257,\\
            -0.2462 \quad &\text{for }N = 2053,
        \end{aligned}
        \right.
    \end{equation*}
    where the infimum of $\log_{2}\epsilon(1,\lambda)$ is numerically approximated via $40,000$ grid points sampled from the set $\Gamma = \{(\eta,\lambda):0<\eta<\frac{1}{2},\frac{1}{2(1-\eta)}<\lambda\le 1\}$. These values are notably higher than the observed maximums from $20,000$ realizations of $\log_2 e_{s,N}^{sh}(\bm{z})$, which are $-2.7858$ for $N = 257$ and $-3.1055$ for $N = 2053$.  This indicates that, in practical applications, the shift-averaged worst-case error obtained by random sampling of the generating vector is significantly smaller than the theoretical bound $\epsilon(\delta,\lambda)$.

\end{remark}

We now explore how taking the median can help to centralize the shift-averaged worst-case error experimentally. We take $k = 11$ and draw $20,000$ realizations of $\log_2[\text{median}(e^{sh}_{s,N}(\bm{z}_1),\ldots,e^{sh}_{s,N}(\bm{z}_k))]$ for randomly chosen $\bm{z}_1,\ldots,\bm{z}_k$. The lower panels in Figure \ref{fig1:distribution} present the results. We find that the distributions of the median of the shift-averaged worst-case errors are more symmetric and are more concentrated around small values. With the use of median trick, we reduced the impact of extremely poor generating vectors $\bm{z}$ on the shift-averaged worst-case error.

\begin{figure}[htbp]
  \centering
  \includegraphics[width=\linewidth]{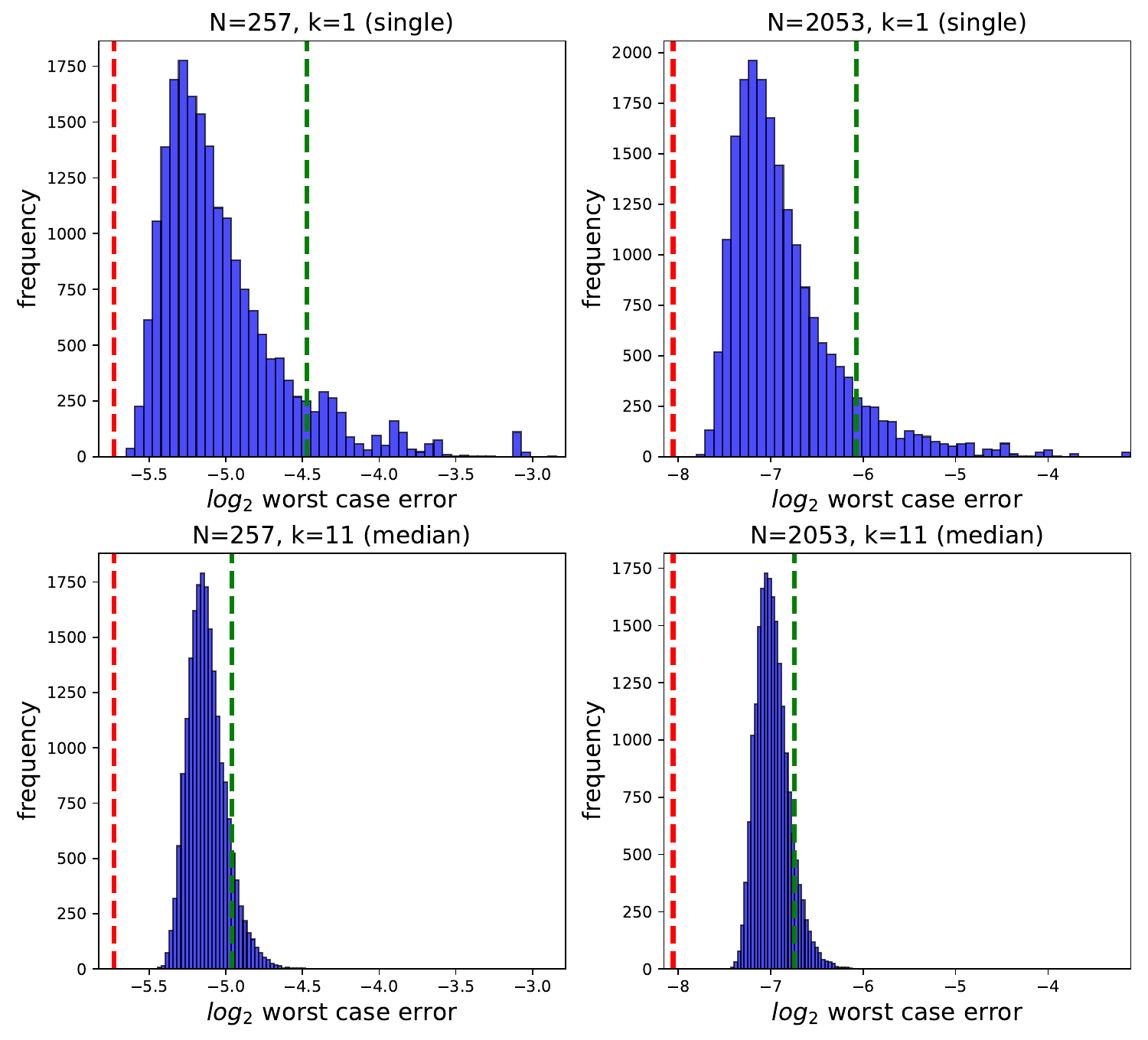}
  \caption{Histograms of the $\log_2$ of the shift-averaged worst-case error $e_{s,N}^{sh}(\bm{z})$ with $\phi(x)=\frac{1}{\sqrt{2\pi}}e^{-\frac{x^2}{2}},\psi_j(x) = e^{-\frac{|x|}{16}},\gamma_u = \prod_{j\in u }\gamma_j,$ and $\gamma_j = \frac{1}{j^2}$ for rank-1 lattice rules with randomly chosen generating vectors with $N = 257$ (left panels) and
  $N = 2053$ (right panels). The upper panels are for a single choice ($k = 1$), while for the lower panels we take the median of the shift-averaged worst-case error for rank-1 lattice rules with $k = 11$ randomly chosen generating vectors. The red dashed line in each histogram represents the shift-averaged worst-case error of the vector generated by the CBC method, and the green dashed line in each histogram represents the $90$th percentile.}
  \label{fig1:distribution} 
\end{figure}

\subsection{Example 2: Pricing an Asian put option with the median QMC method}\label{option}

We consider an arithmetic Asian put option under the Black-Scholes framework. Let the underlying asset price $S_t$ follow the geometric Brownian motion $dS_t = RS_tdt+\sigma S_tdW_t$, where $R>0$ represents the constant risk-free rate, $\sigma > 0$ denotes the volatility, and $W_t$ is the standard Brownian motion. The analytical solution for $S_t$ is given by 
\begin{equation*}
    S_t=S_0\exp((R-\frac{1}{2}\sigma^2)t+\sigma W_t).
\end{equation*}

Given maturity $T$, we discretize the time interval $(0,T]$ using $(d+1)$ equidistant points. The discrete arithmetic average $\Bar{S}$ is defined as
\begin{equation*}
    \Bar{S} = \frac{1}{d+1}\sum_{m=0}^{d}S_{t_m},\quad t_m = \frac{(m+1)T}{d+1},\quad m=0,1,\ldots d.
\end{equation*}
 
Using Principal Component Analysis (PCA) for a Brownian motion construction \cite{glassman2004Montecarlo}, the average price $\Bar{S}=h(\bm{Y})$ can be expressed as the following $(d+1)$-variate function 
\begin{equation*}
 h(\bm{Y})=\frac{1}{d+1}\sum_{m=0}^{d}S_0\exp\left((R-\frac{1}{2}\sigma^2)\frac{(m+1)T}{d+1}+\sigma \bm{A}_m\bm{Y}\right),
\end{equation*}
where $\bm{Y}\sim N(\bm{0},I_{d+1}), \bm{A}_m=(A_{mi})_{i=0}^{d}$, and for $m, i = 0,1,\ldots,d$,
\begin{equation*}
    A_{mi}=\sqrt{\frac{T}{(d+1)(2d+3)}}\frac{\sin\left(\frac{(m+1)(2i+1)\pi}{(2d+3)}\right)}{\sin\left(\frac{(2i+1)\pi}{2(2d+3)}\right)}.
\end{equation*}

We consider the integration problem similar to that in \cite{gilbert2024theory}. The goal is to compute the option value and the probability that the average asset price $\Bar{S}$ falls below a given value $x$. For the option value computation, we have 
\begin{equation}\label{g1}
    \text{value}=e^{-RT}\mathbb{E}\left[\max(K-\Bar{S},0)\right]=e^{-RT}\int_{\mathbb{R}^d+1}\max(K-h(\bm{y}),0)\bm{\phi}_{d+1}(\bm{y})d\bm{y},
\end{equation}
where $K$ denotes the fixed strike price and $\bm{\phi}_{d+1}$ denotes the density function of the $(d+1)$-dimensional standard normal. The probability of $\Bar{S}$ being below $x$ is equivalent to the CDF of $\Bar{S}$ at $x$. Denoting this CDF by $F:\mathbb{R}\to[0,1]$, we compute it via
\begin{equation}\label{g2}
    F(x)=\mathbb{P}(\Bar{S}\le x)=\int_{\mathbb{R}^{d+1}}\text{ind}(x-h(\bm{y}))\bm{\phi}_{d+1}(\bm{y})d\bm{y},
\end{equation}
where $\text{ind}(w)=\bm{1}_{\{w\ge 0\}}$.

Since the integrands in (\ref{g1}) and (\ref{g2}) contain kinks or discontinuities that may reduce the efficiency of the QMC method, we apply the pre-integration technique as in \cite{gilbert2024theory} to smooth these functions. Write $\Tilde{\bm{y}}=(y_1,\ldots,y_d)^\top \in\mathbb{R}^d$. Applying this pre-integration technique allows us to transform the original $(d+1)$-dimensional integrals (\ref{g1}) and (\ref{g2}) into the form
\begin{equation*}
    I_{\bm{\phi}}(P_0g)=\int_{\mathbb{R}^d}P_0 g(\Tilde{\bm{y}})\bm{\phi}_d(\Tilde{\bm{y}})d\Tilde{\bm{y}},
\end{equation*}
where
\begin{equation*}
    P_0 g(\Tilde{\bm{y}})=\left\{
    \begin{aligned}
        &e^{-RT}\int_{-\infty}^{\xi(K,\Tilde{\bm{y}})}(K-h(y_0,\Tilde{\bm{y}}))\phi(y_0)dy_0\quad&\text{for option value},\\
        &\Phi(\xi(x,\Tilde{\bm{y}}))\quad&\text{for CDF $F(x)$},
    \end{aligned}\right.
\end{equation*}
$\phi$ and $\Phi$ denote the density function and the CDF of the $1$-dimensional standard normal, respectively,
and $\xi(x,\Tilde{\bm{y}})$ is the unique solution of the equation $h(\xi,\Tilde{\bm{y}})=x$ for any given $x>0$ and $ \Tilde{\bm{y}}\in\mathbb{R}^d$. It is proven that such $P_0g$ for both cases belongs to the weighted unanchored Sobolev spaces in \cite{gilbert2024theory}. Therefore, we can apply our median QMC method to $P_0g$.

We set the parameters as follows: strike price $K\in\{90,110\}$, CDF at $x\in\{90,110\}$, initial asset price $S_0 = 100$, risk-free rate $r = 0.1$, volatility $\sigma = 0.2$, maturity 
$T = 1$, and $d+1  =16$ time steps. To benchmark the median QMC method, we compare it against two alternatives: the standard MC method and 
the randomly shifted lattice rules with the generating vectors constructed via the CBC method. All three methods are applied to the pre-integrated function $P_0g$.

For the median QMC method, we compute the median of $k = 11$ independent QMC estimators, each utilizing $N$ points.  The CBC method employs weight functions and parameters from \cite{gilbert2024theory}, defined as
\begin{equation*}
    \psi_j^2(x)=\Lambda_0 e^{-2\Lambda_0 |x|}, \text{ with } \Lambda_0=\sigma \sqrt{\frac{T(2d+3)}{d+1}},
\end{equation*}
and
\begin{equation*}
    \gamma_u=\prod_{i\in u}\Lambda_i^{\frac{4}{3}}, \text{ with } \Lambda_i=\frac{\sigma}{2i+1} \sqrt{\frac{T(2d+3)}{d+1}},\quad\text{for all } u\subseteq \{1:d\}.
\end{equation*}
To ensure a fair comparison, all methods use $k\times N$ total function evaluations. Specifically:
\begin{itemize}
    \item The median QMC method takes the median of $k$ independent QMC estimators, each using $N$ points with an independent random shift.
    \item The randomly shifted lattice rule with the CBC method averages $k$ independent QMC estimators, each with $N$ points and an independent random shift.
    \item The MC method directly employs $k\times N$ points.
\end{itemize}

For each $N\in\{17, 31, 67, 127, 257, 521, 1021, 2053, 4099, 8191, 16381, 32771\}$, we estimate the MAE of the three methods by repeating the following procedure:
\begin{enumerate}
    \item Generate $L = 20$ replicates of each estimator.
    \item Compute the exact reference values of $\mathbb{E}[P_0g]$ using $2^{21}$ points from the nested scrambled Sobol' sequence, averaged over $10$ independent repetitions to mitigate potential bias. 
    \item Calculate the MAE via
    
    \begin{equation}\label{MAE}
    \text{MAE}=\frac{1}{L}\sum_{l=1}^{L}|\widehat{P}_0^{(l)}-\mathbb{E}[P_0g]|,
    \end{equation}
    where $\widehat{P}_0^{(l)}$ denotes the estimate from the $l$-th replicate.
\end{enumerate}

\begin{remark}
    We compute the exact reference values via the nested scrambled Sobol' sequence to ensure algorithmic neutrality against the three compared methods, thereby eliminating correlation-induced bias. Since $P_0 g \circ \bm{\Phi}^{-1} \in L^2([0,1]^d)$, according to \cite[Theorem 1]{OWEN1998466}, the QMC estimator constructed from the first $2^m$ points of the nested scrambled Sobol' sequence achieves at least $\mathcal{O}(2^{-m/2})$ convergence rate, with empirically observed superior performance. 

\end{remark}

\begin{figure}[htbp]
  \centering
  \includegraphics[width=\linewidth]{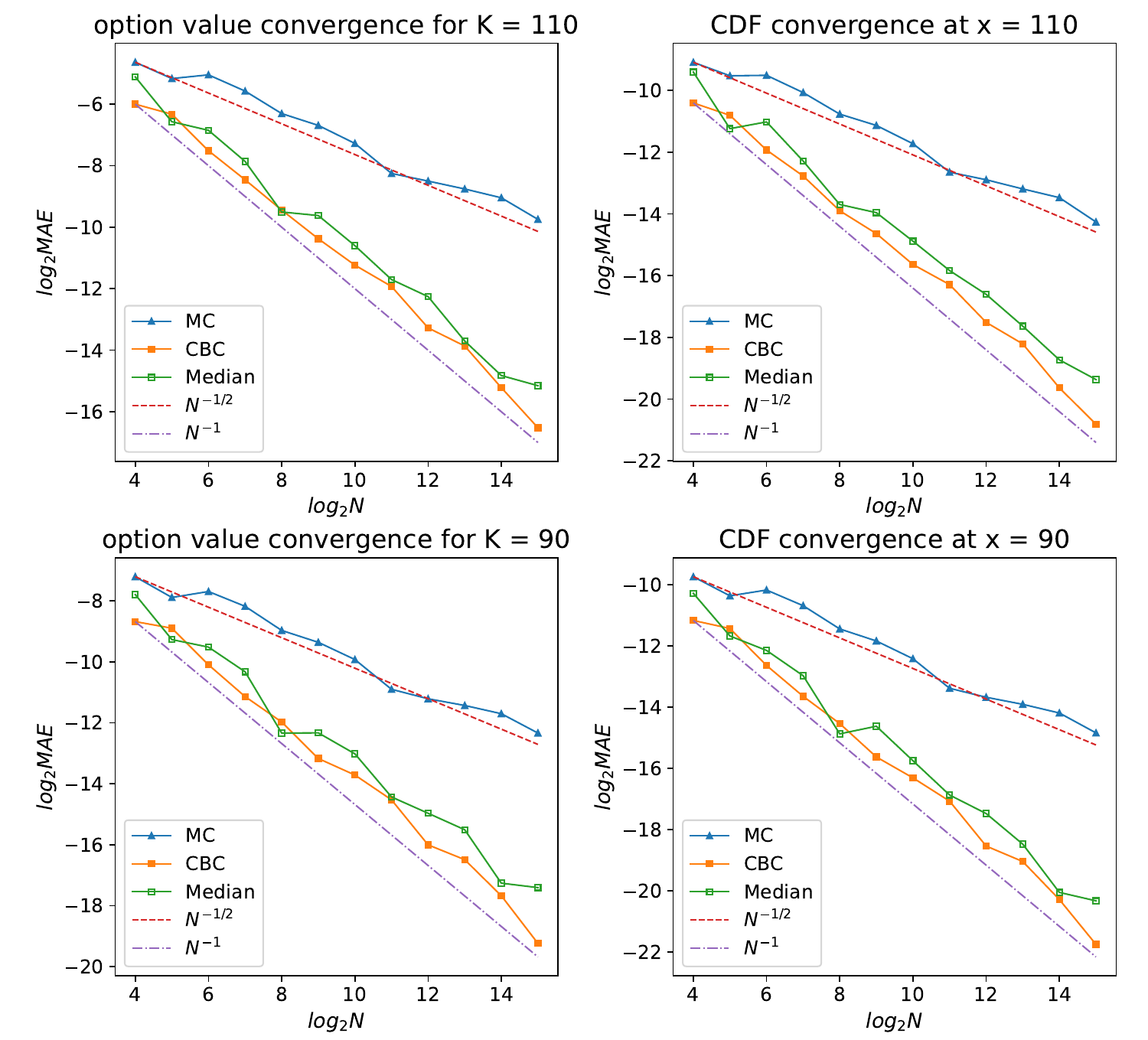}
  \caption{MAE convergence for the MC method,  the randomly shifted lattice rule with the CBC algorithm, and the median QMC method for approximating the option value for $K = 110$ (top left), the option value for $K = 90$ (bottom left), the CDF of $\Bar{S}$ at $x=110$ (top right), and the CDF of $\Bar{S}$ at $x = 90$ (bottom right).  }
  \label{fig2:option} 
\end{figure}
In Figure \ref{fig2:option} we study the convergence in $N$ of the estimated MAEs for the MC method, the CBC method, and the median QMC method. In the left panels we plot the MAEs of the option value for $K = 110$ (top left) and for $K = 90$ (bottom left), while in the right panels we plot the MAEs of the CDF of $\Bar{S}$ at $x = 110$ (top right) and at $x = 90$ (bottom right). 

It is clear from Figure \ref{fig2:option} that the median QMC method can achieve a convergence rate of almost $\mathcal{O}(N^{-1})$ in these cases. Meanwhile, we find that with $k = 11$ independent replications, the median QMC method could achieve an error bound similar to the CBC algorithm. However, by using the median QMC method we do not need to choose the weight parameters  and the weight functions as required by the CBC method, thus obviating the estimation of $\theta_j(\frac{n}{N})$ in (\ref{theta_j}) for certain chosen $\psi_j$.

\subsection{Example 3: Elliptic PDE with log normal random coefficients}

Consider the parametrized ODE

\begin{equation*}
    -\frac{d}{dx}(a^s(x,\bm{y})\frac{du^s(x,\bm{y})}{dx})=1,
\end{equation*}
with homogeneous Dirichlet boundary conditions, $u^s(0,\bm{y})=u^s(1,\bm{y})=0.$
Solving this ODE we obtain
\begin{equation}\label{us}
    u^s(x,\bm{y})=\int_{0}^{x}\frac{c-t}{a^s(t,\bm{y})}dt, \quad        c = \int_{0}^{1}\frac{xdx}{a^s(x,\bm{y})}\left/ \int_{0}^{1}\frac{dx}{a^s(x,\bm{y})}\right..  
\end{equation}
Here we take
\begin{equation*}
    a^s(x,\bm{y})= \exp\left(\sum_{j=1}^{s}\frac{1}{j^2}\sin(2j\pi x)y_j \right),
\end{equation*}
with $y_1,\ldots,y_s\overset{i.i.d.}{\sim}N(0,1)$.

We are interested in computing the expectation  $\mathbb{E}_{\bm{y}}[F(\bm{y})]$, where
\begin{equation*}
F(\bm{y})=G(u^s(\cdot,\bm{y}))=u^s(x_0,\bm{y}), 
\end{equation*}
and $x_0\in\{\frac{1}{3},\frac{2}{3}\}$.
According to \cite{graham2015quasi}, $F$ lies in the weighted unanchored Sobolev space with
\begin{equation}\label{weightf}
    \phi(x)=\frac{1}{\sqrt{2\pi}}e^{-\frac{x^2}{2}},\quad \psi^2_j(x)=e^{-2\alpha_j|x|}, \quad \alpha_j>0.
\end{equation}

We take $s = 30$ and compute the MAEs of the estimators obtained by the MC method, the randomly shifted lattice rule with the CBC algorithm, and the median QMC method. To calculate the integrals in (\ref{us}) for any given $\bm{y}\in\mathbb{R}^s$, we use the 4th-order Gauss-Legendre formula with 200 nodes. Similar to Example 2, the exact values of $\mathbb{E}_{\bm{y}}[F(\bm{y})]$ are estimated by using $2^{21}$ points from the nested scrambled Sobol' sequence averaged over $10$ independent replications. For the median QMC method, we take the median of $k = 11$ independent QMC estimators, each utilizing $N$ points, while for the MC method and the randomly shifted lattice rule with the CBC method, we use $k\times N$ points per method. Furthermore, for the CBC method, we choose the weight parameters and the weight functions as recommended in \cite{graham2015quasi} .
We set $\lambda = 0.55$ and $b_j = \frac{1}{j^2}$ for $j=1,\ldots,s$. For the weight functions $\psi^2_j$ in (\ref{weightf}), we take
\begin{equation*}
    \alpha_1 = \frac{1}{2}\left(b_1+\sqrt{b_1^2+1-\frac{1}{2\lambda}}\right),
\end{equation*}
and
\begin{equation*}
    \alpha_j =\frac{1}{2}\left(b_2+\sqrt{b_2^2+1-\frac{1}{2\lambda}}\right),\quad 2\le j\le s.
\end{equation*}
The weight parameters are assigned as follows
\begin{equation*}
    \gamma_u = \left[\frac{(|u|!)^2}{(\ln 2)^{2|u|}}\prod_{j\in u}\frac{\widetilde{b}_j^2}{(\alpha_j- b_j)\rho_j(\lambda)}\right]^{\frac{1}{1+\lambda}},
\end{equation*}
where
\begin{equation*}
    \rho_j(\lambda)=2\left(\frac{\sqrt{2\pi}\exp(\alpha_j^2/\eta)}{\pi^{2-2\eta}(1-\eta)\eta} \right)^\lambda\zeta(\lambda+\frac{1}{2}),\quad \eta = \frac{2\lambda-1}{4\lambda},
\end{equation*}
and 
\begin{equation*}
    \widetilde{b}_j^2=\frac{b_j^2}{2\exp(b_j^2/2)\Phi(b_j)}.
\end{equation*}
We use the computer programs from the website \url{http://people.cs.kuleuven.be/~dirk.nuyens/qmc4pde/} to obtain the generating vectors for the CBC algorithm. For each $N\in\{2^4,2^5,\ldots,2^{15}\}$, we estimate the MAEs of the three methods with $L = 20$ independent replications similar to (\ref{MAE}).

\begin{figure}[htbp]
  \centering
  \includegraphics[width=\linewidth]{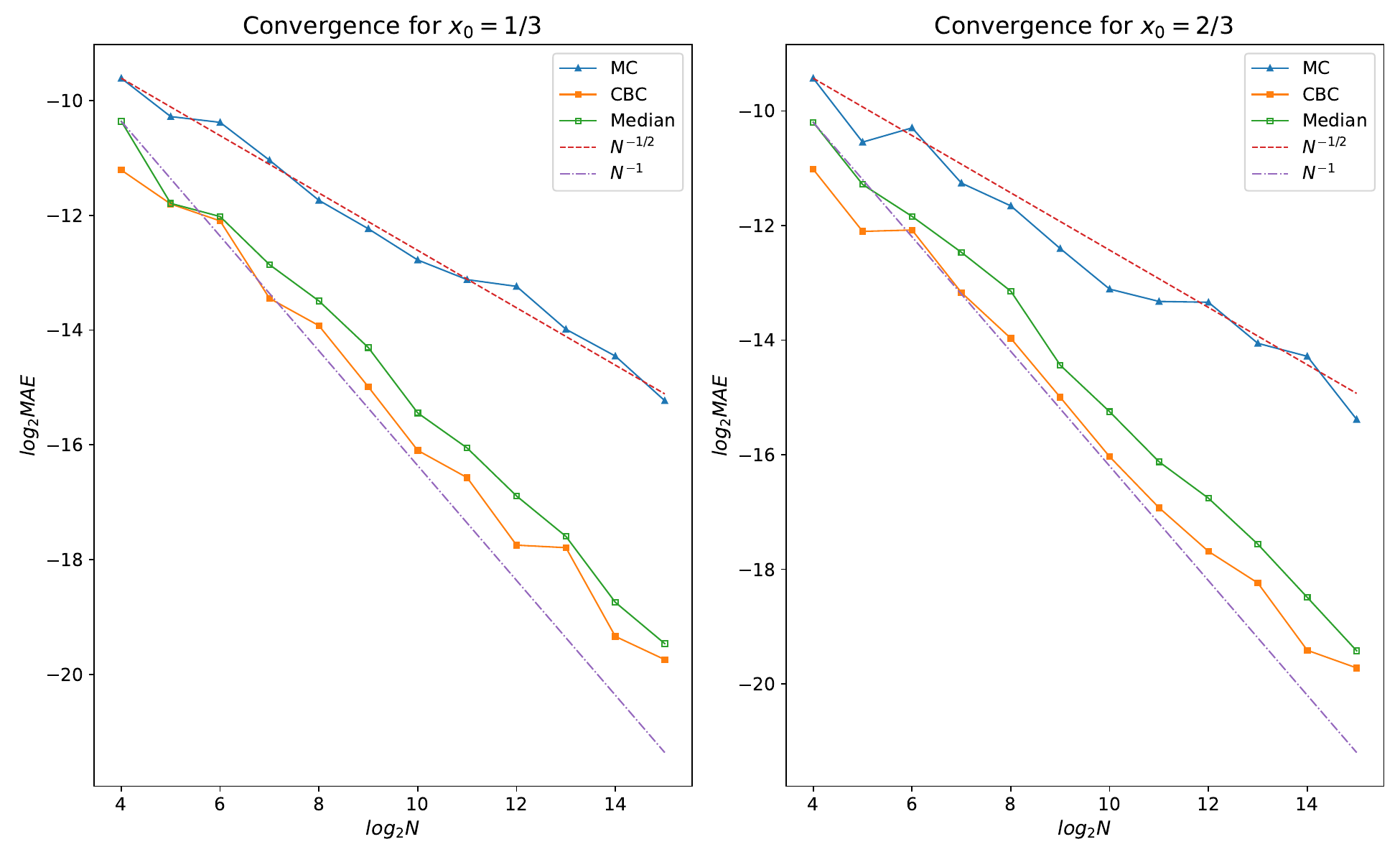}
  \caption{MAE convergence for the MC method, the  randomly shifted lattice rule with the CBC algorithm, and the median QMC method for approximating the expectations for $x_0 = \frac{1}{3}$ (left) and for $x_0 = \frac{2}{3}$ (right).  }
  \label{fig3:PDE} 
\end{figure}

Figure \ref{fig3:PDE} illustrates the convergence in $N$ of the estimated MAEs for the MC
method, the CBC method, and the median QMC method. The left panel is for the MAEs at $x_0 = \frac{1}{3}$ and the right panel is for the MAEs at $x_0 = \frac{2}{3}$. It is clear that both the rank-1 lattice rule with the CBC algorithm and the median QMC method have higher accuracy than the MC method. Meanwhile, we find that with $k = 11$ independent replications, the median QMC method obtains an error bound similar to the CBC algorithm. However, in the CBC algorithm we have to choose the weight parameters and the weight functions carefully, and the selection of the generating vector is time-consuming while the median QMC method avoids these procedures.

\section{Conclusion}\label{conclusion}
In this paper, we extend the construction-free median QMC rule \cite{goda2022construction} to the weighted unanchored Sobolev space $\mathcal{F}$  whose elements are functions
defined over $\mathbb{R}^s$. By taking the median of $k = \mathcal{O}(\log N)$ independent QMC estimators, the median QMC method can achieve an MAE bound of $\mathcal{O}(N^{-r+\epsilon})$, where $r$ is defined in Theorem \ref{bound} and can be arbitrarily close to $1$ for certain pairs of density functions and weight functions. The integration error of the median QMC method achieves the same convergence rate as the integration error of randomly shifted lattice rules obtained by the CBC construction, but does not need to choose the weight parameters and the weight functions in advance. Our numerical experiments support the theoretical results and illustrate that, using the same number of function evaluations, the median QMC method performs comparably to the  
CBC method and outperforms the MC method for option pricing and PDEs with random coefficients.
Note that in this paper we handle the integration over $\mathbb{R}^s$ by the inverse transformation.
It is also desirable to develop similar median tricks for integrations with the truncation method in \cite{dick2018optimal,nuyens2023scaled}. This is left for future work.

\section*{Acknowledgments}
The authors sincerely appreciate the anonymous referees for their valuable suggestions and comments, which significantly improved the quality of this paper.

\bibliographystyle{siamplain}
\bibliography{ref_sim}

\end{document}